\documentclass[letterpaper, 10 pt, conference]{ieeeconf}
\IEEEoverridecommandlockouts                              
\usepackage{amssymb,latexsym,color,amsmath,pifont,graphicx, cite}
\usepackage{subfigure}  
\usepackage{float}
\usepackage[colorlinks=true, allcolors=blue]{hyperref}
\usepackage[nameinlink]{cleveref}
\usepackage{lipsum}
\usepackage{accents}

\newcommand{\widebar}[1]{\bar{#1}}

\DeclareMathOperator*{\minimize}{minimize}

\DeclareMathOperator*{\subjectto}{subject~to}

\newcommand{\reals}[1]{\mathbb{R}^{#1}}
\newcommand{\cN}{\mathcal{N}}
\DeclareMathOperator{\re}{Re}

\DeclareMathOperator{\diag}{\mathrm{diag}}

\newcommand{\norm}[1]{\| #1 \|}    
\newcommand{\inner}[2]{\langle #1,#2\rangle}    

\newcommand{\set}[1]{\{ #1 \}}
\newcommand{\enma}[1]{\ensuremath{#1}}

\newcommand{\non}{\nonumber}

\newcommand{\beq}{\begin{equation}}
\newcommand{\eeq}{\end{equation}}
\newcommand{\ba}{\begin{array}}
\newcommand{\ea}{\end{array}}
\newcommand{\bseq}{\begin{subequations}}
\newcommand{\eseq}{\end{subequations}}

\newcommand{\DefinedAs}[0]{\mathrel{\mathop:}=}
\newcommand{\AsDefined}[0]{=\mathrel{\mathop:}}

\newcommand{\matbegin}{\left[}
\newcommand{\matend}{\right]}

\newcommand{\tbo}[2]{
		\matbegin \begin{array}{c}
				#1 \\ #2
			\end{array} \matend }

\newcommand{\obt}[2]{
	\matbegin \begin{array}{cc}
		#1 & #2
	\end{array} \matend }

\newcommand{\tbt}[4]{
		\matbegin \begin{array}{cc}
				#1 & #2 \\ #3 & #4
			\end{array} \matend }

\newcommand{\ww}{w}

\newcommand{\yy}{y}

\newcommand{\sig}{\sigma}

\newcommand{\xs}{x^\star}
\newcommand{\ws}{\ww^\star}

\newcommand{\dds}{\delta^\star}

\newcommand{\xb}{\widebar{x}}

\newcommand{\hb}{\widebar{h}}
\newcommand{\gb}{\widebar{g}}

\newcommand{\xh}{\widehat{x}}

\newcommand{\yh}{\widehat{\yy}}

\newcommand{\gradh}{\widehat{\nabla f}}
\newcommand{\eh}{\widehat{e}}

\newcommand{\cZ}{\enma{\mathcal Z}}

\newcommand{\cK}{\enma{\mathcal K}}

\newcommand{\cB}{\enma{\mathcal B}}
\newcommand{\cH}{\enma{\mathcal H}}

\newcommand{\cA}{\enma{\mathcal A}}
\newcommand{\cR}{\enma{\mathcal R}}

\newcommand{\sigu}{\widebar{\sigma}}
\newcommand{\sigl}{\underaccent{\bar}{\sigma}}

\newcommand{\cHb}{\enma\widebar{\mathcal H}}
\newcommand{\Deltab}{\enma\widebar{\Delta}}

\newcommand{\bD}{\enma{\mathbb D}}

\newcommand{\hardy}{\enma{\mathcal H}_\infty}
\newcommand{\Rhardy}{\enma{\mathcal{RH}}_\infty}

\newtheorem{theorem}{Theorem}

\newtheorem{lemma}{Lemma}
\newtheorem{remark}{Remark}
\newtheorem{assumption}{Assumption}

\newtheorem{proposition}{Proposition}

\usepackage{empheq}
\usepackage{xcolor}

\newcommand{\vsp}{\vspace*{0.01cm}} 

\begin{document}
%
\title{\bf Automated algorithm design for convex optimization \\ 
problems with linear equality constraints}

\author{Ibrahim~K.\ Ozaslan, Wuwei Wu, Jie Chen, Tryphon T.\ Georgiou, and Mihailo~R.~Jovanovi\'{c}
	\thanks{I.\ K.\ Ozaslan and M.\ R.\ Jovanovi\'{c} are with the Ming Hsieh Department of Electrical and Computer Engineering, University of Southern California, Los Angeles, CA 90089. 
	W.\ Wu and J.\ Chen are with the Department of Electrical Engineering, City University of Hong Kong, Hong Kong SAR, China.	
	T.\ T.\ Georgiou is with the Department of Mechanical and Aerospace Engineering University of California, Irvine, CA 92697. W.\ Wu's and J.\ Chen's work are supported in part by Hong Kong RGC under the project CityU 11207823, CityU 11206024. T.\ T.\ Georgiou's work is supported in part by National Science Foundation (NSF) under grant ECCS-2347357, by Air Office of Scientific Research (AFOSR) under grant FA9550-24-1-0278, and by Army Research Office (ARO) under grant W911NF-22-1-0292.
		{E-mails: ozaslan@usc.edu, w.wu@my.cityu.edu.hk, jichen@cityu.edu.hk, tryphon@uci.edu, mihailo@usc.edu.}
	}
}

\maketitle

\begin{abstract}
Synthesis of optimization algorithms typically follows a {\em design-then-analyze\/} approach, which can obscure fundamental performance limits and hinder the systematic development of algorithms that operate near these limits. Recently, a framework grounded in robust control theory has emerged as a powerful tool for automating algorithm synthesis. By integrating design and analysis stages, fundamental performance bounds are revealed and synthesis of algorithms that achieve them is enabled. In this paper, we apply this framework to design algorithms for solving strongly convex optimization problems with linear equality constraints. Our approach yields a single-loop, gradient-based algorithm whose convergence rate is independent of the condition number of the constraint matrix. This improves upon the best known rate within the same algorithm class, which depends on the product of the condition numbers of the objective function and the constraint matrix.
\end{abstract}


\vspace*{0.15cm}


	\vspace*{-1ex}
\section{Introduction}
	\label{sec.intro}

The design and analysis of optimization algorithms typically begins with selecting an appropriate algorithmic structure based on the problem's optimality conditions. This is followed by a convergence analysis, which often relies on creative and problem-specific reasoning. Such a {\em design-then-analyze\/} approach typically requires deriving tight analytical inequalities$-$an inherently challenging and non-systematic task. This difficulty contributes, in part, to the limited understanding of fundamental performance limits within specific algorithm classes.

\vsp

In his seminal work~\cite{nes03}, Nesterov established lower bounds on the worst-case convergence rates of gradient-based algorithms in the unconstrained setting and proposed algorithms that attain these bounds. Although the lower bounds are derived using an infinite-dimensional setup, the design and analysis of the corresponding algorithms rely on the estimate-sequence technique, which does not easily generalize to broader settings. Given the importance of understanding fundamental performance limits, substantial effort has been directed toward the study of accelerated optimization algorithms~\cite{suboycan16,wibwiljor16,hules17,shidujorsu21}. However, a closer examination reveals that many of these approaches start with a predefined algorithmic structure and use reverse engineering, offering limited insight into the underlying principles that govern optimal performance.

\vsp

Viewing optimization algorithms through the lens of dynamical systems and applying Lyapunov stability theory enables a system-theoretic understanding of their behavior. Although Lyapunov-based analysis is generally regarded as conservative, convergence rates matching known lower bounds can be achieved using Popov-type Lyapunov functions~\cite{muejor19, ozajovAUT25}. Extending such functions to descent-ascent algorithms for constrained optimization or minimax problems, however, remains a significant challenge. In these settings, alternative Lyapunov functions have been proposed~\cite{cornie19, quli18, ozapatjovARX24}, but their design often relies on problem-specific insight. A complementary approach involves fixing the structure of the Lyapunov function and the algorithm to search for valid certificates and algorithm parameters using computational tools~\cite{lesrecpac16, micschebe21}. While this method has the potential of obtaining sharp results~\cite{vanfrelyn17, schebe21}, it falls short of revealing fundamental performance limits due to the fixed forms.

\vsp

Recently, tools and concepts from robust control theory have been employed to demonstrate that Nesterov's lower complexity bound also holds in finite-dimensional settings~\cite{ugrpetsha23}. Building on this insight, a general framework has been developed for the design and analysis of centralized and distributed unconstrained optimization algorithms~\cite{zhawulichegeo24,wuzhalichegeo24,wuchejovgeo24}. This framework models optimization algorithms as transfer functions within a feedback system, derives analytic conditions to guarantee desirable algorithmic properties, and uses Nevanlinna-Pick interpolation to construct transfer functions that meet these conditions. In addition to enabling automated algorithm synthesis, this approach provides a means for identifying fundamental performance limits when algorithm specifications are expressed through necessary and sufficient conditions.

\vsp

In this paper, we build on the framework introduced in~\cite{zhawulichegeo24, wuzhalichegeo24, wuchejovgeo24} to design algorithms for smooth, strongly convex optimization problems subject to linear equality constraints. While such constraints may appear restrictive, they naturally arise in a range of applications, including consensus problems~\cite{boyparryusuh24} and incompressible fluid dynamics~\cite{tahgonsho23}.


\vsp

Among single-loop gradient-based algorithms for convex problems with linear equality constraints, Gradient Descent-Ascent (GDA)~\cite{algsay20} achieves a rate depends on the product of the condition numbers of the strongly convex objective function and the constraint matrix~\cite{algsay20,vansimles23}. Interestingly, with an appropriate stepsize selection, the continuous-time counterpart of GDA can attain a convergence rate that is independent of the condition number of the constraint matrix~\cite{ozajovCDC23}. Although continuous-time guarantees do not directly translate to discrete-time settings, this observation motivates the question of whether the worst-case convergence rate of a discrete-time gradient-based algorithm can be determined solely by the larger of the two condition numbers. In this paper, we take a step toward addressing this question. Leveraging the aforementioned framework, we design a single-loop gradient-based algorithm whose convergence rate depends only on the condition number of the objective function, provided this condition number is larger than that of the constraint matrix.


\vsp

The rest of the paper is organized as follows. In Section~\ref{sec.formulation}, we define the problem, introduce a control-theoretic formulation for algorithm design, and outline the desired algorithm specifications. In Section~\ref{sec.framework}, we present the framework for automated algorithm synthesis within considered algorithm class and obtain the main result. In Section~\ref{sec.numeric}, we use computational experiments to demonstrate the validity of our findings and, in Section~\ref{sec.conc}, we offer concluding remarks.

\section{Problem formulation and a class of algorithms}\label{sec.formulation}

We consider a constrained convex optimization problem,
	\begin{align}
	\label{eq.main}
	\minimize\limits_{x}~f(x)~\subjectto~Ex \, - \, q \,=\, 0
	\end{align}
where $x\in\reals{n}$ is the optimization variable, $f$: $\reals{n}\to\reals{}$ is the objective function, and $E\in\reals{d\times n}$, $q \in \reals{d}$ are parameters in the linear constraint. Without loss of generality, we assume that $E$ is a symmetric positive semidefinite matrix. This assumption is not restrictive because multiplying the equality constraint in~\eqref{eq.main} by $E^\top$ does not change the feasible set. Hence, in our analysis the constraint can be replaced with $E^\top Ex = E^\top q $ without the need to explicitly construct $E^\top E$ in our algorithm (only matrix-vector multiplication involving $E^\top$ and $E$ is required).
\vsp

\begin{assumption}\label{ass.E}
The equality constraint is homogeneous, i.e., $q = 0$. The constraint matrix $E \in\reals{n\times n}$ symmetric and positive semidefinite, i.e., $E=E^\top \succeq0$. 
\end{assumption} 

\vsp

When $q\neq 0$, the homogeneity assumption can be satisfied by writing the constraint as $E^\top E (x - \xb) =0$ where $\xb$ is a feasible solution, and then replacing $x$ in~\eqref{eq.main} with  $x + \xb$. 


We also make the following assumption on the objective function, which is standard in the design of linearly convergent algorithms.  

\vsp

\begin{assumption}\label{ass.str}
The objective function $f$ in~\eqref{eq.main} is $m$-strongly convex with an $L$-Lipschitz continuous gradient $\nabla f$.
\end{assumption} 

\vsp

In the remainder of this section, we utilize the $\cZ$-transform to represent a class of algorithms that we consider and outline design specifications.
  
\subsection{Algorithm representation}

We examine gradient-based optimization algorithms. These can be described by three transfer functions which are (potentially) parameterized by the constraint matrix $E$:
\begin{itemize}
\item $\cK_0(z,E)$ determines the mapping between the optimization variable and the input into a gradient;

\item $\cK_1(z,E)$ determines how the past iterates are used to update the optimization variable;

\item $\cK_2(z,E)$ determines how the past gradients are used to update the optimization variable.
\end{itemize}

	\vsp
	
Let $\xh(z)$ and $\gradh(z)$ be the $\cZ$-transforms of sequences of iterates $\{ x^k \}_{k=0}^\infty$ and gradients $\{ \nabla f(x^k) \}_{k=0}^\infty$. Algorithms that we consider take the following form,
\beq\label{eq.repre}
z\cK_0(z,E)\xh(z)  \,=\, \cK_1(z,E)\xh(z)   \,+\, \cK_2(z,E)\gradh(z).
\eeq
While $\cK_0(z,E)$ is the identity matrix in many cases, it plays a crucial role for accelerated methods such as Nesterov's accelerated algorithm~\cite{nes03} or two-step momentum method~\cite{vanfrelyn17}. In addition to the momentum-based accelerated methods, distributed algorithms such as EXTRA~\cite{shilinwuyin15}, DIGing~\cite{nedolsshi17}, or $\cA\cB\cN$~\cite{xinjakkha19}, can also be brought into the form given by~\eqref{eq.repre}; see~\cite{zhawulichegeo24} for additional details. 

\vsp

We next demonstrate that formulation~\eqref{eq.repre} encompasses not only centralized or distributed algorithms for unconstrained minimization but also primal-dual methods for optimization problems with equality constraints. 

\vsp

As an illustration, let us consider the gradient descent-ascent algorithm associated with~\eqref{eq.main}~\cite{algsay20},
	\bseq	
	\label{eq.gda}
	\begin{align}
	x^{k + 1}  \,&=~ x^k  \,-\, \alpha_1(\nabla f(x^k) \,+\, E^\top y^k )
	\label{eq.gda1}
	\\
	y^{k + 1}  \,&=~  y^k  \,+\,  \alpha_2 E x^k
	\label{eq.gda2}
	\end{align}
	\eseq
where $\alpha_1$ and $\alpha_2$ are the stepsizes. The $\cZ$-transform of~\eqref{eq.gda} can be used to eliminate the dual variable $\yh (z)$,
	\beq
	\left(
	z^2 I - 2 z I + I + \alpha_1 \alpha_2 E^\top E \right) \xh(z)
	\, = \,
	\alpha_1 (1 - z) \gradh (z)
	\non
	\eeq
and the division of this equation with $z$ leads to~\eqref{eq.repre} with,
\begin{align*}
\cK_0(z,E)    \,&=\, I
\\
\cK_1(z,E)  \,&=\, 2 I \,-\, (I \,+\, \alpha_1\alpha_2E^\top E ) z^{-1} 
\\
\cK_2(z,E)  \,&=\,  -\alpha_1(1 \,-\, z^{-1}) I.
\end{align*}

We note that the dual update in~\eqref{eq.gda} can be replaced by its incremental variant (also known as the alternating GDA),
	$y^{k + 1}  = y^k  + \alpha_2 E x^{k + 1}$.
This asynchronous version of a primal-dual algorithm can similarly be represented by~\eqref{eq.repre} with the only change,
$
\cK_1(z,E)  = 2I - \alpha_1\alpha_2E^\top E - z^{-1} I.
$
Furthermore, the primal-dual algorithm based on the proximal augmented Lagrangian~\cite{ozapatjovARX24} associated with~\eqref{eq.main} can be also cast within the same framework. Thus, a broad class of optimization algorithms can be represented using characterization~\eqref{eq.repre}.  


\subsection{Design specifications}
	\label{sec.specs}

We next discuss desired specifications for algorithms that we aim to design to solve optimization problem~\eqref{eq.main}. 

\vsp

\subsubsection{Explicitness} 
To design a single-loop algorithm, we need to avoid circular dependence in the evaluation of $x^{k+1}$. In other words, the computation of $x^{k+1}$ should not exploit $Ex^{k+1}$ or $\nabla f(x^{k+1})$. This requirement can be relaxed if one can afford to utilize linear system solvers~\cite{ozapilari23} or evaluation of proximal operators~\cite{wuchejovgeo24} in the iterative scheme.

\vsp

\subsubsection{Optimality}
The algorithm should asymptotically converge to a solution $\xs$ of problem~\eqref{eq.main}. Under Assumption~\ref{ass.E}, the necessary and sufficient conditions for $\xs$ are given by $\nabla f(\xs) \in\cR(E^\top)$ and $\xs \in \cN(E)$ where $\cR$ and $\cN$ are the range and null spaces of a given matrix. These inclusions can alternatively be characterized by the existence of vectors $\dds_1\in\reals{r}$ and \mbox{$\ws_2\in\reals{n-r}$ such that}
\beq\label{eq.opt}
\nabla f(\xs)  \,=\,  V_1\dds_1, ~~ \xs  \,=\,  V_2\ws_2.
\eeq
Here, $r<n$ determines the rank of the matrix $E$ with the singular value decomposition
\beq\label{eq.svd}
E \,=\, \obt{V_1}{V_2}\tbt{\Sigma}{0}{0}{0}\tbo{V_1^\top  }{V_2^\top  }
\eeq
where $\Sigma = \diag(\sig_1,\dots,\sig_r)$,
$
	\sigu
	\DefinedAs
	\sig_1 \geq
	\cdots
	\geq
	\sig_r
	\AsDefined
	\sigl > 0
$,
and the columns of matrices $V_1\in\reals{n\times r}$ and $V_2\in\reals{n\times n-r}$ form the respective orthonormal bases of $\cR(E^\top)$ and $\cN(E)$. We use $\kappa_E \DefinedAs \sigu/\sigl$ to denote the (effective) condition number of $E$ and $\kappa_f \DefinedAs L/m$ to denote the condition number of the strongly convex objective function $f$. 

\vsp

Because of Assumption~\ref{ass.str}, problem~\eqref{eq.main} has a unique solution, implying uniqueness of $\dds_1$ and $\ws_2$ as well.

\vsp

\subsubsection{Linear convergence}
For a class of optimization problems~\eqref{eq.main} under Assumptions~\ref{ass.E} and~\ref{ass.str}, we require algorithms to linearly converge at a rate of, at least, $\rho$. The linear convergence implies the existence of $\rho\in(0,1)$ and $M>0$ that are independent from the initial condition $x^0$ such that,
\beq\label{eq.linear_conv}
\norm{x^k \,-\, \xs}   \,\leq\, M\norm{x^0 \,-\, \xs} \rho^k.
\eeq 

In contrast to the existing approaches, such as~\cite{boyparryusuh24}, the above framework provides not only asymptotic convergence guarantees but also a, possibly non-conservative, lower bound on the worst-case convergence rate for algorithms that satisfy all design specifications.

	

\section{Automated algorithm synthesis}
	\label{sec.framework}
	
Herein, we formulate the algorithm design as a controller synthesis problem and provide analytical characterization of design specifications in terms of the underlying transfer functions in~\eqref{eq.repre}. By recasting control synthesis as an interpolation problem we provide a solution that leads to an implementable optimization algorithm with a guaranteed convergence rate. 

\subsection{From algorithm to controller design}

The class of optimization algorithms given by~\eqref{eq.repre} can be viewed as a Lur'e system, i.e., a feedback interconnection of an LTI system and a static nonlinear block.  The LTI system in Fig.~\ref{fig.lure} is given by $\xh(z) = \cH(z,E)\gradh(z)$, where
\beq\label{eq.io_map}
\cH(z,E)  \,\DefinedAs\, (z\cK_0(z,E) \,-\, \cK_1(z,E))^{-1}\cK_2(z,E) 
\eeq
and the static nonlinear block takes the form
$
\Delta( e^k )  \DefinedAs \nabla f(e^k  + \xs)  - \nabla f(\xs) 
$
with the error defined as $e^k \DefinedAs x^k - \xs$. Under Assumption~\ref{ass.str}, nonlinearity $\Delta$ satisfies
$m\norm{e}^2 \leq \inner{\Delta(e)}{e}  \leq L\norm{e}^2$, for all $e \in \reals{n}$ thereby implying sector-boundedness of $\Delta$ in $[m,L]$.

\vsp

In contrast to a conventional interpretation, the plant in this control system is represented by a static nonlinear block~$\Delta$. Our goal is to design a linear controller $\cH(z,E)$, which represents the optimization algorithm, to track step inputs $\tfrac{z}{z-1} \, \xs$ with unknown magnitudes $\xs$ while rejecting the step disturbances $\tfrac{z}{z-1} \, \nabla f(\xs)$ with unknown magnitudes $\nabla f(\xs)$ for all possible sector-bounded plants $\Delta$ in~$[m,L]$. The controller is parameterized by a constraint matrix $E$ of a rank $r<n$ with singular values $\sig_i\in[\sigl,\sigu]$, $i=1,\ldots,r$.  

\vsp

Before discussing design specifications, we employ a standard technique to transform the Lur'e system into a form that is convenient for analysis/synthesis.

\vsp
\begin{figure*}[t]
	\centering
	\begin{tabular}{c@{\hspace{0.2 cm}}c@{\hspace{1cm}}c@{\hspace{-0.2 cm}}c}
		\subfigure[]{\label{fig.lure}}
		&
		\subfigure[]{\label{fig.decom}}
		&
		&
		\subfigure[]{\label{fig.exp}}
		\\[-.2cm]
		\begin{tabular}{c}
			\includegraphics[width=0.29\textwidth]{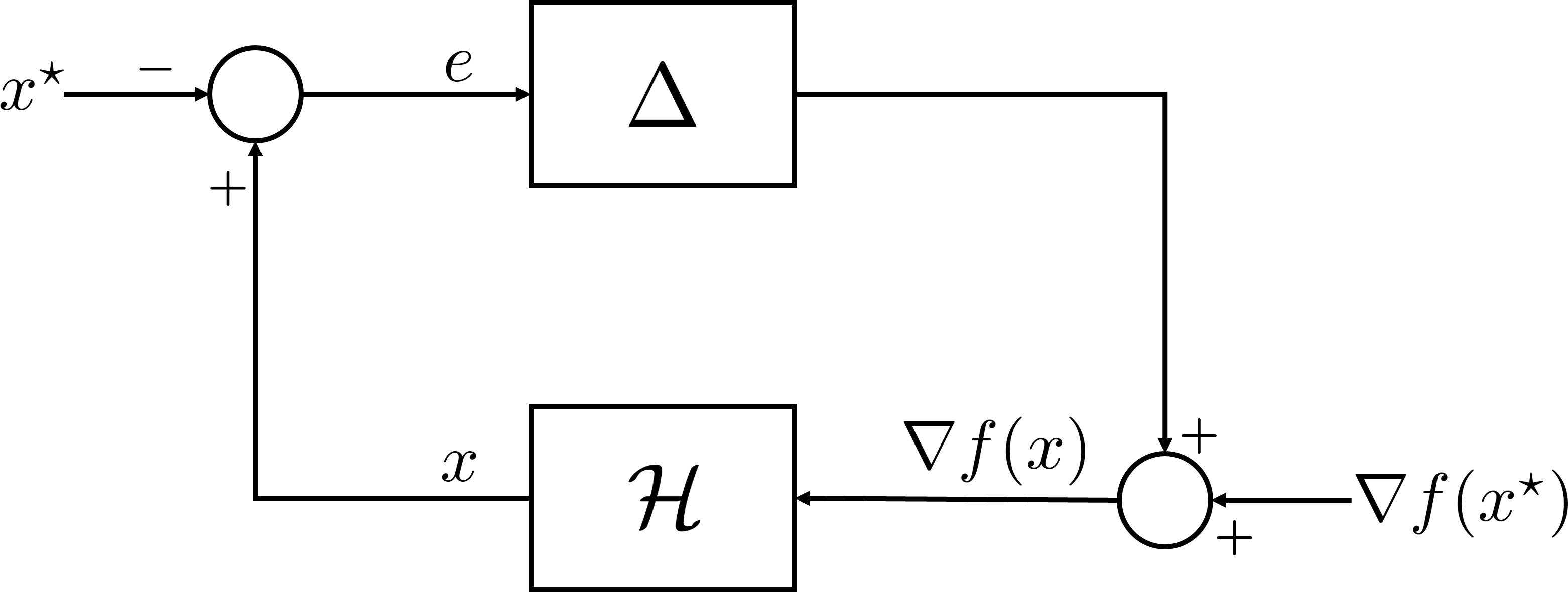}
		\end{tabular}
		&
		\begin{tabular}{c}
			\includegraphics[width=0.29\textwidth]{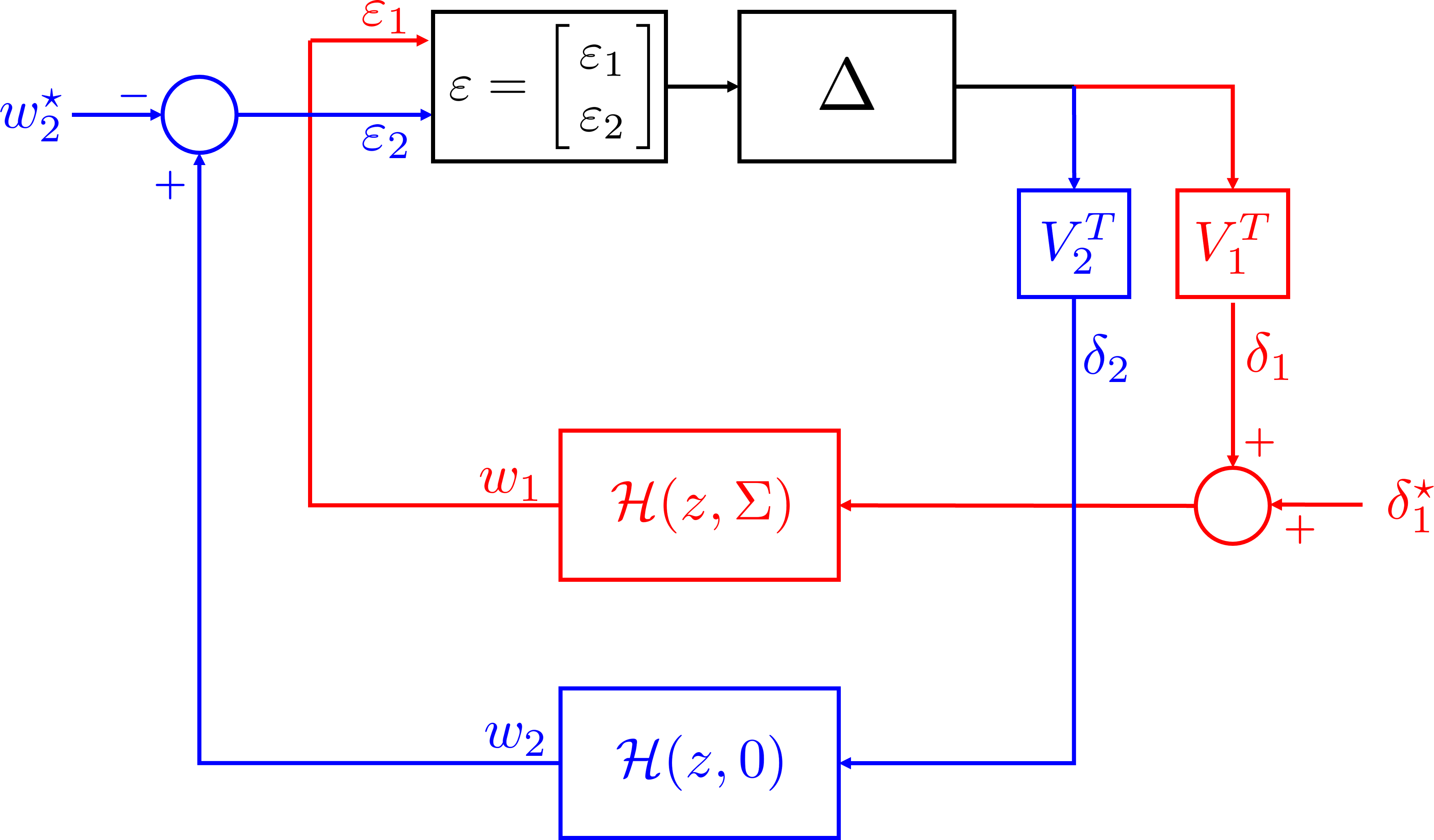}
		\end{tabular}
		&\hspace{-1cm}
		\begin{tabular}{c}
			\vspace{.25cm}
			\normalsize{\rotatebox{90}{$\log_{10}\norm{x^k \,-\, \xs}/\norm{\xs}$}}
		\end{tabular}
		&
		\begin{tabular}{c}
			\includegraphics[width=0.29\textwidth]{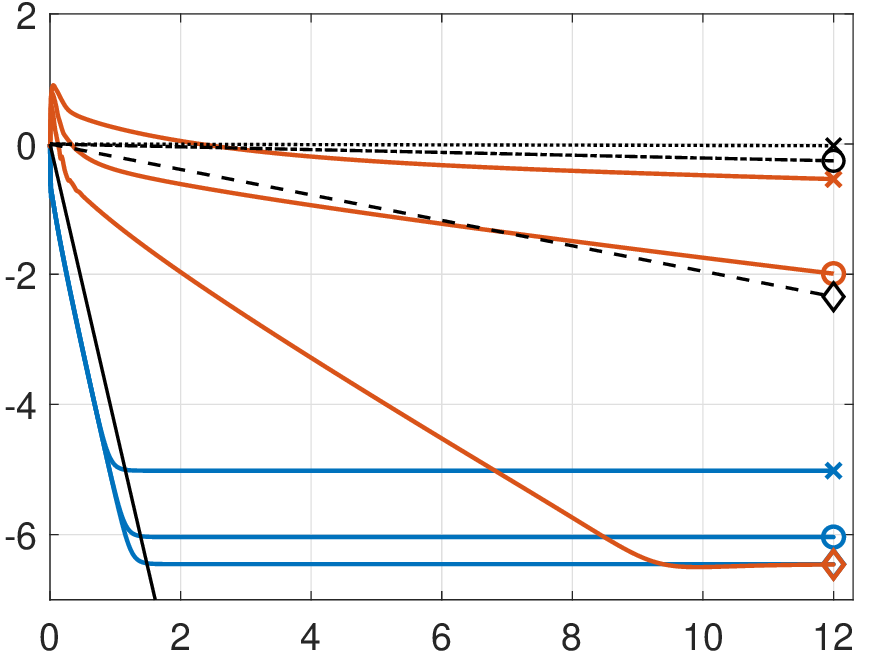}
			\\[-0.1 cm]  {$k~(\times10^4)$}
		\end{tabular}	
	\end{tabular}
	\vspace{-0.2cm}
	\caption{(a) Lur’e system representing the optimization algorithms. (b) Decomposition of
the Lur'e system into two subsystems coupled by the nonlinear element. (c) Convergence of algorithm~\eqref{eq.algo} (blue) and the incremental GDA method~\eqref{eq.gda} (red) for problem~\eqref{eq.main}. The lines with markers show convergence of algorithms in the first (diamond), second (circle), and third (cross) problem instances with the respective constraint matrices $E_1$,  $E_2$, and $E_3$. While the condition number of the objective function is set to $\kappa_f = 2000$, the (effective) condition numbers of the constraint matrices are $\kappa_{E_1} = 10$, $\kappa_{E_2} = 100$, and $\kappa_{E_3} = 1000$. The solid black line marks the theoretical convergence rate $\rho_{\mathrm{syn}}$ in~\eqref{eq.theorem.result}; the black dash, dash-dot, and dot lines show the rate $\rho_{\mathrm{gda}}$ in~\eqref{eq.gda_rate} for $i=1,2,3$, respectively.}
	\label{fig.traj}
\end{figure*}

\subsubsection{Coordinate transformation}

The change of variables using the singular vectors of $E$, 
\beq\label{eq.coordinate}
\ww  \,\DefinedAs\,  V^\top x \,=\, [V_1~V_2]^\top x \,=:\, [\ww_1^\top~\ww_2^\top]^\top   
\eeq
together with the definitions
\beq\non
\begin{split}
\epsilon  \,&=\,  V^\top e  \,=\,  \tbo{\!\!\ww_1\!\!}{\!\!\ww_2  - \ws_2\!\!},
\\
 \delta  \,&=\, V^\top \Delta(V\epsilon)  \,=\,  \tbo{\!\!\delta_1\!\!}{\!\!\delta_2\!\!},
\end{split}~~
\begin{split}
\ws \,&=\, V^\top\xs \,=\, \tbo{\!\!0\!\!}{\!\!\ws_2\!\!}
\\
\dds  \,&=\, V^\top \nabla f(\xs)  \,=\,  \tbo{\!\!\dds_1\!\!}{\!\!0\!\!}
\end{split}
\eeq
can be used to cast the transfer function in Fig.~\ref{fig.lure} as, 
\beq\non
\cH(z,E)  \,=\,  \obt{V_1}{V_2}\tbt{\cH(z, \Sigma)}{0}{0}{\cH(z, 0)}\tbo{V_1^\top  }{V_2^\top  }.
\eeq 
The resulting two linear systems are coupled by the nonlinear block, as shown in Fig.~\ref{fig.decom}. Since these subsystems are in the diagonal form, they can be further split into the scalar systems, each of which shares the same scalar transfer function parameterized by either singular values or zero, 
\beq\non
\cH(z,\Sigma)   \,=\,  \diag\big(\{h(z,\sig_i)\}_{i=1}^r\big), ~~ \cH(z,0)   \,=\,  h(z,0)I_{n-r}.
\eeq

\vsp

Our task is to design a scalar transfer function $h(z,\sigma)$ with parameter $\sigma$ taking an arbitrary value within the interval $[\sigl,\sigu]$ or being equal to zero.



\subsubsection{Loop transformation}

To make nonlinear component $\Delta$ sector-bounded in $[0,\infty)$, we apply loop transformation~\cite[Sec.\ 6.5]{kha02} to the feedback interconnection in Fig~\ref{fig.lure}. This transformation enables the use of standard passivity-based stability criteria without introducing additional conservatism. Specifically, the transformation takes the pair $(\cH, \Delta)$ with sector-bounded $\Delta$ in $[m,L]$ as input and gives the pair $(\cHb, \Deltab)$ with transformed nonlinearity $\Deltab$ being sector-bounded in $[0,\infty)$ as output. 
The one-to-one correspondence between the two transfer functions is given by
\beq\label{eq.loop}
\cH(z,E)  \,=\, (m\cHb(z,E) \,-\, LI)^{-1}(\cHb(z,E) \,-\, I).
\eeq
In contrast to $\cH$, $\cHb$ has to be a stable transfer function to guarantee stability of the transformed Lur'e system; this is because $\Deltab$ is sector-bounded in $[0,\infty)$.

\vsp

Using coordinate transformation~\eqref{eq.coordinate}, we can also diagonalize the transformed linear system,
	\beq
	\non
	\cHb(z,\Sigma)   
	\,=\,  
	\diag\big(\{\hb(z,\sig_i)\}_{i=1}^r\big), 
	~
	\cHb(z,0)   
	\,=\,  
	\hb(z,0)I_{n-r}
	\eeq
with the following one-to-one relation between $h$ and $\hb$,
\beq\label{eq.onetoone}
h(z,\sigma)  \,=\, (\hb(z,\sigma) \,-\, 1)/(m\hb(z,\sigma) \,-\, L).
\eeq 

\subsection{Analytical characterization of design specifications}

We next obtain conditions for transfer function $\hb(z,\sigma)$ to satisfy design specifications outlined in Section~\ref{sec.specs}.

\vsp

\subsubsection{Explicitness as causality}

A necessary condition for explicitness is the strict causality of the linear controller in the Lur'e system. The controller is strictly causal when transfer function $\cH(z,E)$ is strictly proper, i.e, $\cH(\infty,E) = 0$. Using the correspondence in~\eqref{eq.loop}, this is equivalent to $\cHb(\infty, E)  =  I$, and the scalar transfer function should satisfy
	\bseq
	\label{eq.interpolation}
	\beq\label{cond.causal}
	\hb(\infty,\sigma)  \,=\,  1,
	~~\forall \, \sigma \,\in\,[\sigl,\sigu]\,\cup\,\set{0}. 
	\eeq

\subsubsection{Optimality as input tracking and disturbance rejection}
The asymptotic convergence of the optimization algorithm to the optimal solution is equivalent to the asymptotic tracking of the step input $\tfrac{z}{z-1} \, \xs$ and rejection of the step disturbance $\tfrac{z}{z-1} \, \nabla f(\xs)$ in the Lur'e system. To achieve this goal, the subsystems $\cH(z,\Sigma)$ and $\cH(z,0)$ must have a blocking zero and a pole at $z=1$, respectively. This follows from internal model principle~\cite{frawon76}. For exogenous signals~\eqref{eq.opt}, Proposition~\ref{theorem.imp} establishes that these conditions are not only sufficient but also necessary for input tracking and disturbance rejection.

\vsp

\begin{proposition}\label{theorem.imp}
Let the feedback interconnection in Fig.~\ref{fig.lure} be stable for any sector bounded nonlinearity $\Delta$ in $[m,L]$. The error asymptotically converges to zero, i.e., $\lim_{k\to\infty} e^k = 0$, if and only if the transformed scalar transfer function satisfies
	\beq
	\label{cond.opt}
	\hb(1,\sig)  \,=\, \left\{
	\ba{ll}
	1, &\sig \,\in\, [\sigl,\sigu]
	\\
	L/m, &\sig  \,=\, 0.
	\ea
	\right.
	\eeq
\end{proposition}

\vspace*{1ex}

\begin{proof}
See Appendix.
\end{proof}
%

\vsp

\subsubsection{Linear convergence as $\rho$-stability}
A necessary and sufficient condition for linear convergence of the synthesized algorithm is the $\rho$-stability of the Lur'e system in Fig.~\ref{fig.lure}. We use the circle criterion~\cite[Sec.\ 7.1]{kha02} to certify the $\rho$-stability. 

\vsp

The circle criterion implies $\rho$-stability of the transformed Lur'e system $(\cHb, \Deltab)$, with sector-bounded nonlinear block $\Deltab$ in $[0,\infty)$, if the scalar transfer function $h(\gamma z,\sigma)$ is {\em strictly positive real\/} for all $\gamma\in(\rho, 1]$. This can be equivalently expressed as,
\beq\label{cond.stab}
\ba{rl}
\hspace{-2ex}\mathrm{i}) \!&\!\!\! \hb(\gamma z,\sigma)\in\Rhardy
\\[0.15cm]
\hspace{-2ex}\mathrm{ii})\!&\!\!\! \re \, ( \hb(\gamma z,\sig) ) > 0, ~\forall z \in \bD^c,~\sigma \in[\sigl,\sigu]\cup\set{0}.  
\ea
\eeq
\eseq
for all $\gamma\in(\rho, 1]$, where $\Rhardy$ denotes set of real, rational, stable transfer functions,  $\re (\cdot)$ is the real part of a complex number, and $\bD\DefinedAs\set{z\in\mathbb{C}, \, |z|<1}$.

\subsection{From controller design to interpolation}
	\label{sec.NP}

Design of a scalar transfer function $\hb$ that satisfies causality~\eqref{cond.causal}, tracking~\eqref{cond.opt}, and $\rho$-stability~\eqref{cond.stab} conditions can be cast as the following interpolation problem:
\beq
\tag{P1}
\label{eq.P1}
\ba{c}
\text{Find a $\rho\in(0,1)$ and design $\hb(z,\sig)$ such that}
\\
\text{conditions in~\eqref{eq.interpolation} hold for every $\gamma\in(\rho,1].$}
\ea
\eeq 

Direct application of the Nevanlinna-Pick interpolation technique~\cite[Ch.\ 9]{doyfratan13} to problem~\eqref{eq.P1} is not possible because of conditional statement~\eqref{cond.opt} for different values of $\sig$. Hence, following~\cite{zhawulichegeo24}, we adopt a greedy approach to obtain a solution as outlined below.

\vsp

Since interpolation condition~\eqref{cond.opt} includes two cases ($\sig\in[\sigl,\sigu]$ or $\sig = 0$), we first construct a candidate solution $(\rho_1, \hb_1(z,\sig))$ to~\eqref{eq.P1} with conditions~\eqref{eq.interpolation} applied only to $\sig\in[\sigl,\sigu]$. Similarly, we obtain another candidate solution $(\rho_2, \hb_2(z))$, ensuring that conditions~\eqref{eq.interpolation} hold for $\sig = 0$. We then combine these to obtain a solution to problem~\eqref{eq.P1} by setting $\rho  =  \max(\rho_1, \rho_2)$ and defining $ \hb(z,\sig) = \hb_1(z,\sig)$ while enforcing $\hb(z,0) = \hb_2(z)$. The last step leverages the additional degrees of freedom offered by the Nevanlinna-Pick interpolation. Lemmas~\ref{lemma.first_case} and~\ref{lemma.second_case} provide a solution to problem~\eqref{eq.P1} for $\sig\in[\sigl,\sigu]$ and $\sig=0$, respectively.

\vsp

\begin{lemma}\label{lemma.first_case}
All transfer functions satisfying conditions~\eqref{eq.interpolation} for all $\sig\in[\sigl,\sigu]$  and $\gamma\in(0,1]$ are given by
\beq\label{eq.lemma.first_case}
	\hb_1(\gamma z,\sig)  \,=\,  \dfrac{z(z \,-\, \gamma) \,+\, \gb(z,\sig)(\gamma z \,-\, 1)}{z(z \,-\, \gamma) \,-\, \gb(z,\sig)(\gamma z \,-\, 1)}
\eeq
where $\gb(z,\sig)\in\Rhardy$ and $\norm{\gb(z,\sig)}_\infty<1$.
\end{lemma}

\begin{proof}
The proof is a direct application of Nevanlinna-Pick interpolation; see~\cite[Prop.\ 1]{zhawulichegeo24} for details.
\end{proof}

\vsp

\begin{lemma}\label{lemma.second_case}
A transfer function satisfying conditions~\eqref{eq.interpolation} for $\sig=0$
and $\gamma\in(\rho,1]$ exists if and only if $\rho \geq (L-m)/(L+m)$. Furthermore, for $\rho = (L-m)/(L+m)$, the transfer function that satisfies all four conditions is unique and is given by, $\hb_2(\gamma z)  = (\gamma z + \rho)/(\gamma z - \rho)$.
\end{lemma}

\vsp

\begin{proof}
See~\cite[Theorem 5]{zhawulichegeo24}.
\end{proof}

\vsp

We combine two candidate solutions as described above: we set $\hb(z,\sig) = \hb_1(z,\sig)$ and determine $\gb(z,\sig)$ in Lemma~\ref{lemma.first_case} to ensure $\hb(z,0) = \hb_2(z)$. This yields $\gb(z,0) = (\rho - z)/(\rho z - 1)$ where  $\rho = (L-m)/(L+m)$. This brings interpolation problem~\eqref{eq.P1} to:
\beq
\tag{P2}
\label{eq.P2}
\!
\ba{l}
\text{Design $g(z,\sig)$ such that}
\\[0.1cm]
\left\{
\ba{ll}
\! \gb(z,\sig)\in\hardy\text{ and }\norm{\gb(z,\sig)}_\infty  \,<\,1, & \sig \,>\,  0
\\
\! \gb(z,0)  \,=\, -(z \,-\, \rho)/(\rho z \,-\, 1) , & \sig  \,=\,  0.
\ea 
\right.
\ea
\eeq

A solution to problem~\eqref{eq.P2} is given by
\beq\label{eq.sol_p2}
\gb(z,\sig)  \,=\,  -((1\,-\,\sig)z \,-\, \rho)/(\rho z \,-\, (1\,-\,\sig))
\eeq
with the restriction $\sig\in(2/(1 + \kappa_f), 1]$. This solution is obtained by moving the pole at $1/\rho$ of $\gb(z,0)$ inside the unit circle by ensuring that $\cH_\infty$ norm of $\gb(z,\sig)$ is less \mbox{than one}.

\subsection{Algorithm synthesis}

We are now ready to cast the transfer functions obtained in Section~\ref{sec.NP} in terms of their state-space realizations, thereby leading to an implementable optimization algorithm. 


\vsp

Substitution of $\hb$ given by~\eqref{eq.lemma.first_case} into~\eqref{eq.onetoone} yields the following transfer function for $h$,
	\beq\non
	h(z,\sig)  
	\,=\, 
	\tfrac{\tfrac{2}{L \,+\, m} \, g(z,\sig)}{z \,+\, g(z,\sig)},~~ g(z,\sig)  
	\,\DefinedAs\,  
	\rho\gb(z/\rho,\sig)\tfrac{z \,-\, 1}{z \,-\, \rho^2}.
	\eeq
Furthermore, the definition of $h$ in~\eqref{eq.io_map} allows us to determine transfer functions $k_0$, $k_1$, and $k_2$: Setting $k_0(z,\sig) = 1$ gives $
	k_1(z,\sig) 
	= -g(z,\sig)
	$, and
	$
	k_2(z,\sig) 
	= 
	2g(z,\sig)/(L+m) 
	$. As a result, we have
\bseq\label{eq.k}
\begin{align}
&\cK_0(z,E) \,=\, I
\\
&\cK_1(z,E)  
	\,=\, 
	\tfrac{z - 1}{z - \rho^2} \, 
	(zI - W)^{-1}\big(Wz - \rho^2I\big)
	\label{eq.k1}
\\
&\cK_2(z,E)  \,=\, -(1 - \rho) \, \cK_1(z,E)
\label{eq.k2}
\end{align}
\eseq
where $W \DefinedAs I - E$. The remaining task is to obtain a state-space realization for the input-output map~\eqref{eq.io_map} based on the transfer functions given in~\eqref{eq.k}. 

Since $\cK_0 = I$ and $\cK_2$ is a scaled version of $\cK_1$, we have
	\beq
	z\xh (z)
	\, = \, \cK_1 (z, E) (\xh(z) \,-\, (1 \, - \, \rho) \gradh (z) )
	\eeq
Taking the inverse $\cZ$-transform of~\eqref{eq.k} yields the following recursion,
\bseq\label{eq.algo}
\begin{align}
u^k \,&=\,x^k \,-\, (1-\rho)\nabla f(x^k) 
\\
y^{k}  \,&=\,  W(y^{k-1} + u^{k}) - \rho^2 u^{k-1} 
\\
x^{k+1}   \,&=\, \rho^2 x^k \,+\, y^{k} \,-\, y^{k-1} 
\end{align}
\eseq
with the initial condition $y^{-1} = u^{-1} = 0$. 

	\vsp

\begin{theorem}\label{theorem.result}
Let Assumptions~\ref{ass.E} and~\ref{ass.str} hold. If the singular values of the constraint matrix $E$ belong to the interval $(2/\kappa_f, 1]$, where $\kappa_f = L/m$, then algorithm~\eqref{eq.algo} converges linearly to the unique solution of~\eqref{eq.main} with a rate of at least,  
	\beq\label{eq.theorem.result}
	\rho_{\mathrm{syn}}  \,=\,  1  \,-\,  2/(\kappa_f \,+\, 1).
	\eeq
\end{theorem} 
\vspace*{1ex}

\begin{proof}
	The proof directly follows from the synthesis procedure presented in Section~\ref{sec.NP}.
\end{proof}
\vspace*{1ex}

	\begin{remark}
While the upper bound on the singular values can easily be satisfied by scaling the equality constraint with its largest column-sum or its trace, the lower bound requires the condition number of $E$ to be smaller than half the condition number of the objective function, $ \kappa_f/2 = L/(2m)$. 
	\end{remark}
		
	\vsp
	
	\begin{remark}
For problem~\eqref{eq.main} with a rectangular constraint matrix $E\in\reals{d\times n}$, both algorithm~\eqref{eq.algo} and incremental GDA~\eqref{eq.gda} perform only one gradient computation and one matrix-vector multiplication with $E$ and $E^\top$ per iteration. However, while~\eqref{eq.algo} has a realization with $3 n$ states, the incremental GDA has $n+d$ states. 
	\end{remark}
	
\section{Computational experiments}
	\label{sec.numeric}

We demonstrate the merits and the effectiveness of our algorithm for the quadratic objective function in~\eqref{eq.main},
	\beq
	f(x) 
	\; = \; 
	\tfrac{1}{2} \, x^\top Q x \, + \, p^\top x. 
	\non
	\eeq	
For the problem with $n = 100$ optimization variables and $d = 100$ constraints, we randomly select problem data $Q\in\reals{n\times n}$, $E \in\reals{d \times n}$, and $p \in \reals{n}$ with $q = 0$ in~\eqref{eq.main}. The eigenvalues of $Q$ are scaled such that $L = \lambda_{\max} (Q) = 2000$ and $m = \lambda_{\min} (Q) = 1$. We generate three instances of constraint matrices $E_1$, $E_2$, and $E_3$ with the corresponding singular values scaled such that the largest singular values are the same and equal to $\sig_1 = 1$, but the smallest singular values take different values $\sig_r=0.1$, $\sig_r=0.01$, and $\sig_r=0.001$, respectively, with $r = 80$. As a result, the (effective) condition numbers of the constraint matrices are determined by $\kappa_{E_1} = 10$, $\kappa_{E_2} = 100$, and $\kappa_{E_3} = 1000$.

\vsp

We perform computational experiments using our algorithm~\eqref{eq.algo} and the incremental GDA method~\eqref{eq.gda} to compare and contrast the influence of conditions numbers of constraint matrices on the corresponding convergence rates. We set all initial conditions to zero. To optimize the best known convergence rate of the GDA method~\cite{algsay20, vansimles23}, 
	\beq\label{eq.gda_rate}
	\rho_{\mathrm{gda}}  \,=\, \sqrt{1 \,-\, \tfrac{1}{\kappa_f} \! \left(\tfrac{1}{\kappa_E} 
	\,-\, 
	\tfrac{1}{\kappa^2_E}\right) }
	\eeq 
we set $\alpha_1 = (1 - \kappa_E^{-1})/L$ and $\alpha_2 = m/\sig_1$ in~\eqref{eq.gda}. As illustrated in Fig.~\ref{fig.exp}, the convergence rate of our algorithm is invariant under changes in condition number of the constraint matrix, closely aligning with the theoretical bound. In contrast, the convergence rate of the GDA algorithm significantly deteriorates as the condition number of $E$ increases. Furthermore, as suggested by the comparison of the convergence rates~\eqref{eq.theorem.result} and~\eqref{eq.gda_rate}, our algorithm converges substantially faster than the GDA method in all computational experiments.  


	\vspace*{-1ex}
\section{Concluding remarks}
	\label{sec.conc}
	
In this work, we leveraged robust control techniques for the automated synthesis of optimization algorithms. We developed a single-loop, gradient-based algorithm for solving strongly convex optimization problems with linear equality constraints. When the condition number of the constraint matrix is smaller than that of the objective function, we demonstrated that the algorithm's convergence rate depends solely on the objective function's condition number. Under this assumption, our approach offers a significant improvement over existing methods$-$such as GDA$-$whose rates depend on the product of both condition numbers. Promising future directions include relaxing the assumption on the constraint matrix, characterizing the fundamental performance limits of gradient-based methods, and extending the framework to a wider class of constrained optimization problems.

	\vspace*{-1ex}
\bibliographystyle{IEEEtran}

\appendix
\textit{Proof of Proposition~\ref{theorem.imp}:}
Sufficiency follows from the internal model principle. For necessity, we consider the following equation representing the feedback interconnection in Fig.~\ref{fig.lure},
\begin{align*}
\eh(z) \,+\, \tfrac{z}{z \,-\, 1} \, \xs  
	\; = \;
	\cH(z,E) \big( \widehat{\Delta}(z) \,+\, \tfrac{z}{z \,-\, 1} \, \nabla f(\xs) \big)
\end{align*}
where $\eh(z)$ and $\widehat{\Delta}(z)$ are the $\cZ$-transforms of the sequences $\set{e^k}_{k=0}^\infty$ and $\set{\Delta(e^k)}_{k=0}^\infty$, respectively. Relation~\eqref{eq.loop} yields
\begin{align*}
(m\cHb(z,E) \,&-\, LI)\big((z \,-\, 1)\eh  \,+\,  z\xs\big)
\\
\,&=\, (\cHb(z,E) \,-\, I)\big((z \,-\, 1)\widehat{\Delta}(z)  \,+\,  z\nabla f(\xs)\big).
\end{align*} 
By assumption, the feedback interconnection is stable for any sector-bounded nonlinearity $\Delta$ in $[m,L]$, which implies that the transfer function $\cHb(z,E)$ does not have any poles on the unit circle, i.e., $\cHb(1,E)\neq \infty$. Therefore, evaluating the above equation at $z=1$ results in
\beq\non
(m\cHb(1,E) \,-\, LI) \xs
\,=\, (\cHb(1,E) \,-\, I)\nabla f(\xs).
\eeq
Coordinate transformation~\eqref{eq.coordinate} in conjunction with optimality conditions~\eqref{eq.opt} gives
$$
V_2(m\cHb(1,0) - LI) \ws_2
\,=\, V_1(\cHb(1,\Sigma) - I)\dds_1
$$
which must hold for any $\dds_1\in\reals{r}$ and $\ws_2\in\reals{n-r}$. Hence,
\beq\non
\cHb(1,0)  \,=\, (L/m)I_{n-r},~~\cHb(1,\Sigma)  \,=\, I_r.
\eeq
\end{document}